\documentclass[10pt,reqno]{article}
\usepackage{amsthm,amsmath,amssymb,a4wide}
\usepackage{color}
\newtheorem{theorem}{Theorem}

\newtheorem{lemma}[theorem]{Lemma}

\theoremstyle{definition}
\newtheorem{definition}[theorem]{Definition}
\newtheorem{remark}[theorem]{Remark}
\newtheorem{example}[theorem]{Example}

\numberwithin{theorem}{section}
\numberwithin{equation}{section}

\begin{document}
\title{ Signed Partitions and Rogers-Ramanujan type Identities}

\author{Abdulaziz M. Alanazi, Augustine O. Munagi, Andrew V. Sills}
\date{\today}
\maketitle

\begin{abstract}
George Andrews [\emph{Bull. Amer. Math. Soc.}, 2007, 561--573] introduced the idea of a
\emph{signed partiton} of an integer; similar to an ordinary integer partitions, but where some of the
parts could be negative.   Further, Andrews reinterpreted the classical G\"ollnitz--Gordon 
partition identities in terms of signed partitions.  In the present work, we provide interpretations
of the sum sides of Rogers--Ramanujan type identities, including a new signed partition 
interpretation of the G\"ollnitz--Gordon identities, different from that of Andrews.
Both analytic and bijective proofs are presented.
\end{abstract}

\section{Introduction}\label{introd}
We begin with a classical definition.
\begin{definition}
A \emph{partition} $\lambda$ of an integer $n$ is a nonincreasing finite sequence of positive integers 
$(\lambda_1, \lambda_2, \dots, \lambda_\ell)$ that sum to $n$.  Each $\lambda_i$ is called a \emph{part} of the
partition $\lambda$.  Here, $n$ is called the \emph{size} of $\lambda$ and may be denoted $|\lambda| = n$.
The \emph{length} $\ell = \ell(\lambda)$ of $\lambda$ is the number of parts in $\lambda$.
\end{definition}  
For example, the partitions of $3$ are $(1,1,1)$, $(1,2)$ and
$(3)$.    

G. E. Andrews~\cite{Andrews2007} introduced the notion of \emph{signed partition}, where, unlike with
ordinary partitions, some of the parts could be negative integers.   In~\cite{Andrews2007}, he reinterpreted some classical
partition identities in terms of signed partitions.

Let us formally define a signed partition as follows:
\begin{definition}
A {\em signed partition} $\sigma$ of an integer $n$ is a pair $(\pi, \nu)$ of ordinary partitions such that $|\pi| - | \nu | = n$.  The members of $\pi$ (resp. $\nu$) are the \emph{positive parts}
(resp. \emph{negative parts}) of the signed partition $\sigma$.
\end{definition}
Thus, $\sigma = \big( (1,1,2,3,3,3), (1,2,3) \big)$ is a signed partition of $13-6 = 7$.
Alternatively, we could choose to write $\sigma$ as $(1,1,2,3,3,3,-1,-2,-3)$.

In the sequel, we will provide signed partition theoretic interpretations of some classical partition identities.
For example, the combinatorial version of the first Rogers--Ramanujan identities is as follows:
let $RR_1(n)$ denote the number of (ordinary) partitions of $n$ in which the difference between parts is at least $2$.
Then $RR_1(n)$ equals the number of (ordinary) partitions of $n$ into parts congruent to $1$ or $4 \pmod{5}$.
We will show in Theorem~\ref{SRR_1} below that $RR_1(n)$ also equals
the number of signed partitions of $n$ in which smallest positive part is even, positive parts differ by at least by $3$ and alternate in parity, and negative parts are at most the number of positive parts. 
Analogous results are presented for other Rogers--Ramanujan type identities.

\section{Preliminary results}

The following lemma will be used to establish several of our identities combinatorially.
\begin{lemma}\label{lembit}
Let $B=(b_1,b_2,\ldots)$ be a nonempty binary sequence with at least one non-zero entry. There is a unique partition of least weight associated with $B$ which is gap-free with smallest part $1$ (ignoring a possible initial string of 0's).
\end{lemma}
\begin{proof}
The partition, denoted by $t(B)=(t_1,t_2,\ldots)$, is obtained as follows.
Set $t_1=b_1$, then for any index $j>1$, 
\begin{equation}\label{eqbit}
t_j =\begin{cases} t_{j-1}& \text{if}\ b_j=b_{j-1}\\
t_{j-1}+1 & \text{if}\ b_j\neq b_{j-1}.
\end{cases}
\end{equation}
Observe that $t_j\leq j$ for all $j$, and $B\equiv t(B)$ (mod 2) term-wise.

For example, if $B=(0,1,1,0,1,0,0)$ then using \eqref{eqbit} we obtain $t(B)=(0,1,1,2,3,4,4)$.
\end{proof}

\bigskip

The ordinary partitions of $n$ will be shown to be equinumerous with a certain class of 
signed partitions of $n$ as set forth in the following theorem.

\begin{theorem}\label{thmpn}
Let $p(n)$ denote the number of ordinary partitions of $n$ and let $p_{-1}(n)$ denote the number of signed partitions of $n$ in which positive parts alternate in parity with smallest 
part even, and negative parts are distinct and at most the number of positive parts. Then 
    \begin{equation}
        p(n)=p_{-1}(n).
    \end{equation}
\end{theorem}
Here and throughout, we shall use the standard notation $(a;q)_n := \prod_{i=0}^{n-1} (1-aq^i)$.
Also, we mention that $q$ may be regarded as a formal variable, or alternatively as a complex
variable as long as $|q| < 1$.
\begin{proof}[Generating function proof]
    \begin{equation}
        \sum_{n=0}^{\infty}p(n)q^n=\sum_{n=0}^{\infty}\frac{q^n}{(q;q)_n}=\sum_{n=0}^{\infty}\frac{q^n(-q;q)}{(q^2;q^2)_n}=\sum_{n=0}^{\infty}\frac{q^{2+3+\cdots+(n+1)}(-q^{{-1}};q^{-1})}{(q^2;q^2)_n}=\sum_{n=0}^{\infty}p_{-1}(n)q^n,
    \end{equation}
where the first equality is attributed to Euler in~\cite[p. 19, Eq. (2.2.5) with $t=q$]{Andrews1976}.    
\end{proof}

\begin{proof}[Bijective Proof]
Let $p[n]$ and $p_{-1}[n]$ denote the corresponding sets of ordinary and signed partitions respectively.
We describe a map $f: p[n]\rightarrow p_{-1}[n]$.
Let $\lambda\in p[n]$ where $\lambda = (\lambda_1,\lambda_2,\ldots,\lambda_k)$ with $1\leq \lambda_1\leq \lambda_2\leq \cdots \leq \lambda_k$.
Then
\begin{equation}\label{eqmap}
f(\lambda)=\begin{cases} \lambda & \text{if $\lambda_i\equiv i-1$ (mod 2)}\ \forall\, i\\
\gamma & \text{otherwise},
\end{cases}
\end{equation}
where $\gamma$ is an signed partition of $n$ obtained as follows. Let $B_k$ denote the unique parity-alternating binary sequence of length $k$ with first term 0, $B_k=(0,1,0,1,\ldots)$.
\begin{itemize}
\item[(i)] Define $A :=(\lambda+B_k)$ (mod 2);
\item[(ii)] Use Lemma \ref{lembit} to obtain $t(A)$ and denote its conjugate by $t(A)'$;
\item[(iii)] Set $\gamma$ to be the signed partition whose positive parts consist of $\lambda+t(A)$ and whose negative part sizes consist of $t(A)'$.  
\end{itemize}

Conversely, let $\gamma\in p_{-1}[n]$. If $\gamma$ has no negative parts, then $f^{-1}(\gamma)=\gamma$. Otherwise denote the sets of the positive parts and the  negative part sizes of $\gamma$ by $U$ and $V$ respectively. Then $f^{-1}(\gamma) = U-V'$, where $V'$ may be padded with initial 0's if necessary.

It is clear that $A\equiv t(A)$ (mod 2) termwise, by construction. Thus if $\gamma=(\gamma_1,\ldots,\ldots,\gamma_k)$ and $t(A)=(t_1,\ldots,\ldots,t_k)$, then for any index $j$, we have $\gamma_j\equiv \lambda_j+t_j\equiv \lambda_j+(\lambda_j+(j-1)\, \text{mod 2})\equiv j-1$ (mod 2). Since $t_k\leq k$ and $t(A)$ is gapfree with 1, the conjugate $t(A)'$ has distinct parts with largest part at most $k$. Thus in step (iii), $\gamma\in p_{-1}[n]$. So $f$ is a bijection. 

\vskip 5pt
For example, consider $\lambda=(1, 1, 1, 2, 3, 6, 10, 10, 16)\in p[50]$. Then $A\equiv \lambda+(0,1,0,1,0,1,0,1,0)\equiv (1, 0, 1, 1, 1, 1, 0, 1, 0)$ (mod 2). So $t(A)=(1, 2, 3, 3, 3, 3, 4, 5, 6)$, $t(A)'=(1,2,3,7,8,9)$ and $\lambda+t(A) = (2, 3, 4, 5, 6, 9, 14, 15, 22)$. Hence $f(\lambda)=(2, 3, 4, 5, 6, 9, 14, 15, 22,-1,-2,-3,-7,-8,-9)\in p_{-1}[50]$.
\end{proof}

\begin{theorem}\label{SD}
Let $D(n)$ denote the number of partitions of $n$ into distinct parts and let $D_{-1}(n)$ denote the number of signed partitions of $n$ in which positive parts are even and distinct, and negative parts are at most the number of positive parts. Then \begin{equation}
        D(n)=D_{-1}(n).
    \end{equation}

\end{theorem}
\begin{proof}
   \begin{multline}
       \sum_{n=0}^\infty D(n)q^n=\sum_{n=0}^\infty \frac{ q^{n(n+1)/2} }{(q;q)_{n}}=\sum_{n=0}^\infty \frac{ q^{1+2+\cdots+n} q^{1+2+\cdots+n} (-q^{-1};q^{-1})_n }{(q^2;q^2)_{n}} \\ =\sum_{n=0}^\infty \frac{ q^{2+4+\cdots+(2n)} (-q^{-1};q^{-1})_n }{(q^2;q^2)_{n}}=\sum_{n=0}^\infty D_{-1}(n)q^n
   \end{multline}
\end{proof}

\begin{proof}[Bijective Proof of Theorem \ref{SD}] The proof is essentially the same as that of Theorem \ref{thmpn}. But we give sufficient details below.

Let $D[n]$ and $D_{-1}[n]$ denote the corresponding sets of ordinary and signed partitions of $n$ respectively.
Consider the map $f1: D[n]\rightarrow D_{-1}[n]$.
Let $\lambda\in D[n]$ where $\lambda = (\lambda_1,\lambda_2,\ldots,\lambda_k)$ with $1\leq \lambda_1<\lambda_2<\cdots <\lambda_k$.
Then
\begin{equation}\label{eqmap}
f1(\lambda)=\begin{cases} \lambda & \text{if $\lambda_i$ is even}\ \forall\, i\\
\gamma & \text{otherwise},
\end{cases}
\end{equation}
where $\gamma$ is a signed partition of $n$ obtained as follows. 
\begin{itemize}
\item[(i)] Set $A :=\lambda$ (mod 2);
\item[(ii)] Use Lemma \ref{lembit} to obtain $t(A)$ and denote its conjugate by $t(A)'$;
\item[(iii)] Set $\gamma$ to be the signed partition whose positive parts consist of $\lambda+t(A)$ and whose negative part sizes consist of $t(A)'$.  
\end{itemize}

Conversely, let $\gamma\in D[n]$. Then $f1^{-1}(\gamma)=f^{-1}(\gamma)$ (see the proof of Theorem \ref{thmpn}). 

Observe that the only difference between this map and $f$ is the omission of $B_k = (0,1,0,1,\ldots)$ in step (i). However, the techniques are identical since $f1$ here implicitly uses $B_k=(0,0,0,\ldots)$ because the positive parts here are all even.

Note that if $\lambda_i = \lambda_{i+1}$ for some $i$, then it is immediate that $\gamma_i = \gamma_{i+1}$, a contradiction. The converse also holds.
\vskip 3pt

For example, consider $\lambda=(1, 2, 4, 5, 13, 14)\in D[39]$. Then $A\equiv (1,0,0,1,1,0)$ (mod 2). So $t(A)=(1,2,2,3,3,4)$, $t(A)'=(1,3,5,6)$ and $\lambda+t(A) = (2,4,6,8,16,18)$. Hence $f1(\lambda)=(2,4,6,8,16,18,-1,-3,-5,-6)\in D_{-1}[39]$.
\end{proof}

\section{The Rogers--Ramanujan identities}\label{rogramid}
The Rogers--Ramanujan identities, in their analytic form, are due to L. J. Rogers in~\cite{R94}:
\begin{equation} \label{RR1an}
\sum_{n=0}^\infty \frac{q^{n^2}}{(q;q)_n} = \prod_{n=0}^\infty \frac{1}{(1-q^{5n+1})(1-q^{5n+4})}.
\end{equation}
MacMahon~\cite{M18} and Schur~\cite{S17} independently realized that the left hand side of
\eqref{RR1an} was the generating function for $RR_1(n)$, the number of partitions of $n$ in which the difference between parts is at least $2$.  We provide an alternative interpretation in terms
of signed partitions.

\begin{theorem}\label{SRR_1}
Let $RR_{-1}(n)$ denote the number of signed partitions of $n$ in which positive parts differ by at 
least by $3$ and alternate in parity with the smallest part even, and negative parts are at most the 
number of positive parts. Then 
   \begin{equation}
       RR_1(n)=RR_{-1}(n)
   \end{equation} 
\end{theorem}
\begin{proof}[Generating function proof]
   \begin{multline}
       \sum_{n=0}^\infty RR_1(n)q^n=\sum_{n=0}^\infty \frac{ q^{n^2} }{(q;q)_{n}}=\sum_{n=0}^\infty \frac{ q^{1+3+\cdots+(2n-1)} q^{1+2+\cdots+n} (-q^{-1};q^{-1})_n }{(q^2;q^2)_{n}} \\ =\sum_{n=0}^\infty \frac{ q^{2+5+\cdots+(3n-1)}  (-q^{-1};q^{-1})_n }{(q^2;q^2)_{n}}=\sum_{n=0}^\infty RR_{-1}(n)q^n
   \end{multline}
\end{proof}

\begin{proof}[Bijective Proof of Theorem \ref{SRR_1}] 
The proof is exactly the same as that of Theorem \ref{thmpn}. 
Let $RR_1[n]$ and $RR_{-1}[n]$ denote the corresponding sets of ordinary and signed partitions respectively. Then $f: RR_1[n]\rightarrow RR_{-1}[n]$ gives the required bijection.

If $\lambda_i-\lambda_{i-1}=1$ for some $i$, it may be verified that $\gamma_i-\gamma_{i-1}=1$, a contradiction.
Conversely if the latter holds, then $\lambda_i-\lambda_{i-1} = (\gamma_i-t_i)-(\gamma_{i-1}-t_{i-1})=1-(t_i-t_{i-1})\leq 1$.

\end{proof}

\begin{example}
Consider $\lambda=(1, 4, 6, 11, 14, 16)\in RR_1[52]$. Then $A\equiv (1,0,0,1,0,0) +$ 

$(0,1,0,1,0,1)\equiv (1,1,0,0,0,1)$ (mod 2). So $t(A)=(1,1,2,2.2,3)$, $t(A)'=(1,4,6)$ and $\lambda+t(A) = (2,5,8,11,16,19)$. Hence $f(\lambda)=(2,5,8,11,16,19,-1,-4,-6)\in RR_{-1}[52]$.
\end{example}
 
\begin{theorem}\label{SRR_2}
Let $RR_2(n)$ denote the number of partitions of $n$ into parts $>1$ and difference between parts is at least $2$ and let $RR_{-2}(n)$ be the number of signed partitions of $n$ in which positive parts differ by at least $3$ and alternate in parity and have no 1's  with the smallest part odd, and negative parts are at most the number of positive parts. Then 
\begin{equation}
  RR_2(n)=RR_{-2}(n)
\end{equation} 
\end{theorem}

\begin{proof}[Generating function proof]
   \begin{multline}
       \sum_{n=0}^\infty RR_2(n)q^n=\sum_{n=0}^\infty \frac{ q^{n^2+n} }{(q;q)_{n}}=\sum_{n=0}^\infty \frac{ q^{2+4+\cdots+2n} q^{1+2+\cdots+n} (-q^{-1};q^{-1})_n }{(q^2;q^2)_{n}} \\=\sum_{n=0}^\infty \frac{ q^{3+6+\cdots+3n} (-q^{-1};q^{-1})_n }{(q^2;q^2)_{n}}\sum_{n=0}^\infty RR_{-2}(n)q^n
   \end{multline}
\end{proof}

\begin{proof}[Bijective Proof of Theorem \ref{SRR_2}] 
The proof is almost identical with that of Theorem \ref{SRR_1}. The only change required is to define $B_k$ as $B_k=(1,0,1,0,\ldots)$. 

Note that $1\in \lambda$ if and only if $1\in\gamma$, and if   $\lambda_i-\lambda_{i-1}<2$ for some $i$, then it may be verified that $t_i=t_{i-1}$; thus $\lambda_i+t_i-\lambda_{i-1}-t_{i-1} = \gamma_i-\gamma_{i-1}<2$, a contradiction.
\end{proof}

\section{The G\"ollnitz--Gordon identities}\label{gollnitz}
The G\"ollnitz--Gordon identities were discovered independently by H. G\"ollnitz~\cite{G67} and
B. Gordon~\cite{G65}.  The first G\"ollnitz--Gordon identity states that the number of partitions
of $n$ such that all parts differ by at least two and no consecutive even numbers appear as parts equals the number of partitions into parts congruent to $1$, $4$, or $7$ modulo $8$.

In~\cite{Andrews2007}, Andrews finds a class of signed partitions that are equinumerous with
partitions counted in the first G\"ollnitz--Gordon identity:

\begin{theorem}[Andrews]\label{AFGG}
Let $GG_1(n)$ denote the number of partitions of $n$ in which differences between parts are at least $2$ and at least $4$ if the parts are even. Let $GG_{-1}(n)$ denote the number of signed partitions of $n$ where the positive parts are even and at least twice the number of positive parts, the negative parts are odd, distinct and at most twice the number of positive parts. Then
\begin{equation} \label{An1}
GG_{-1}(n)=GG_1(n).
\end{equation}
\end{theorem}

Let $GG_1[n]$ and $GG_{-1}[n]$ denote the sets of objects enumerated by $GG_1(n)$ and $GG_{-1}(n)$, respectively.  The third author~\cite{Sills2008} provided a bijective map $h:GG_{-1}[n]\rightarrow GG_1[n]$. Let $(\gamma_1,\gamma_2,\cdots,\gamma_j)\in GG_1[n]$ and $(\pi_1,\pi_2,\cdots,\pi_j)\in GG_{-1}[n]$ and define $h$ by \begin{equation} \label{eqsills1}
h:\, (\gamma_1,\gamma_2,\cdots,\gamma_j)\rightarrow (\pi_1,\pi_2,\cdots,\pi_j,-\wp(\gamma_1)\times1,-\wp(\gamma_2)\times3,\cdots,-\wp(\gamma_j)\times(2j-1)),
\end{equation}
where $\wp(k)$ is defined as 
 \[ \wp(k)  = \begin{cases}
  1 & \mbox{ if $ \text{ $k$ is odd} $}\\
  0 & \mbox{ if \text{ $k$ is even}},
  \end{cases}
  \]
and for each $k$ we have $$\pi_k=\gamma_k+4k-2j-2+\wp(\gamma_k)+2\sum^j_{i=k+1}\wp(\gamma_i).$$

The inverse map $h^{-1}$ is then defined as follows:\\
Let $\pi=(\pi_1,\pi_2,\cdots, \pi_r,-f_1\cdot 1,-f_2\cdot 3,\cdots , -f_r\cdot (2r-1))\in GG_{-1}[n]$ such that 
\begin{equation}\label{eqchar1}
f_j  = \begin{cases}
  1 & \mbox{ if $-(2j-1)\in \pi$}\\
  0 & \mbox{ if $-(2j-1)\notin \pi$}.
  \end{cases}
\end{equation}
Then $\pi$ has the form
\begin{equation} \label{eqsills2}
\pi=(\pi_1,\pi_2,\cdots, \pi_r,-f_1\cdot 1,-f_2\cdot 3,\cdots , -f_r\cdot (2r-1))\rightarrow (\gamma_1,\gamma_2,\cdots, \gamma_r).
\end{equation}
Therefore,
$$\gamma_j=\pi_j-4j+2r+2-f_j-2\sum_{i=j+1}^rf_i.$$

\begin{example} Let $\lambda=(20,17,15,12,9,7,4,1)\in GG_1[85]$. Then using \eqref{eqsills1} we obtain $\pi_1=20+4(1)-2(8)-2+0+2(5)=16$, $\pi_2=17+4(2)-16-2+1+2(4)=16$, \ldots, $\pi_8=1+4(8)-16-2+1+2(0)=16$. Thus the image is $(16^8,0(-1),1(-3),1(-5),0(-7),1(-9),1(-11),0(-13),1(-15))=(16^8,-3,-5,-9,-11,-15)\in GG_{-1}[85].$\\
Conversely, applying $h^{-1}$ to $(16^8,-3,-5,-9,-11,-15)$ we obtain $\gamma_1=16-4(1)+2(8)+2-0-2(5)=20$, $\gamma_2=16-4(2)+2(8)+2-1-2(4)=17$, \ldots, $\gamma_8=16-4(8)+2(8)+2-1-2(0)=1$. Thus the pre-image is $(20,17,15,12,9,7,4,1)\in GG_1[85]$.    
\end{example}

Define $\ell^+$ to be the number of positive parts in a signed partition. We state  another interpretation of $GG_1(n)$ below.
\begin{theorem}\label{SFGG} 
    Let $GG'_{-1}(n)$ denote the number of signed partitions of $n$ in which positive parts are even and differ by at least $4$, and negative parts are odd, distinct and at most $2\ell^+-1$. Then 
\[GG'_{-1}(n)=GG_1(n).\]
\end{theorem}

\begin{proof}
    
\begin{align*}
\sum_{n=0}^\infty GG_1(n)q^n
&= 1+\sum_{n=1}^\infty \frac{q^{n^2} (-q;q^2)_n}{(q^2;q^2)_n}\\
&= 1+\sum_{n=1}^\infty \frac{q^{n^2} (1+q)(1+q^3)\cdots(1+q^{2n-1})}{(1-q^2)(1-q^4)\cdots(1-q^{2n})}\\
&= 1+\sum_{n=1}^\infty \frac{q^{n^2} q^{n^2}(\frac{1}{q}+1)(\frac{1}{q^3}+1)\cdots(\frac{1}{q^{2n-1}}+1)}{(1-q^2)(1-q^4)\cdots(1-q^{2n})}\\
&= 1+\sum_{n=1}^\infty \frac{q^{2+6+\cdots+(4n-2)}(\frac{1}{q}+1)(\frac{1}{q^3}+1)\cdots(\frac{1}{q^{2n-1}}+1)}{(1-q^2)(1-q^4)\cdots(1-q^{2n})}\\
&= \sum_{n=0}^\infty GG'_{-1}(n)q^n. 
\end{align*}
\end{proof}
 
\begin{proof}  
\emph{Bijection}.
 Let $GG'_{-1}[n]$ be the set of signed partitions of $n$ enumerated by $GG'_{-1}(n)$. We define a map $g:GG_1[n]\rightarrow GG'_{-1}[n]$: if  $\lambda=(\lambda_1,\lambda_2,\cdots, \lambda_r)\in GG_1[n]$, then 
\begin{equation}\label{eqsills3}
g:\, (\lambda_1,\lambda_2,\cdots, \lambda_r)\rightarrow (\tau_1,\tau_2,\cdots, \tau_r,-\wp(\lambda_1)\cdot 1,-\wp(\lambda_2)\cdot 3,\cdots , -\wp(\lambda_r)\cdot (2r-1)),
\end{equation}
where 
$$\tau_j=\lambda_j+\wp(\lambda_j)+2\sum_{i=j+1}^r\wp(\lambda_i),\ 1\leq j\leq r.$$

Note that $\tau_j\equiv \lambda_j+\wp(\lambda_j)\equiv 0$ (mod 2) for all $j$. Furthermore, for a fixed $j$ we have $\pi_j - \pi_{j+1} = \lambda_j - \lambda_{j+1} + \wp(\lambda_j) + \wp(\lambda_{j+1})\geq 4$ because $\lambda_j - \lambda_{j+1}\geq 2,3,4$ and $\wp(\lambda_j) + \wp(\lambda_{j+1}) = 2,1,0$ when $\lambda_j$ and $\lambda_{j+1}$ are both odd, have opposite parities or both even, respectively.

The inverse map $g^{-1}$ is obtained as follows:\\
Let $\tau=(\tau_1,\tau_2,\cdots, \tau_r,-f_1\cdot 1,-f_2\cdot 3,\cdots , -f_r\cdot (2r-1))\in GG'_{-1}[n]$, where the $f_j$ are defined for $\tau$ as in \eqref{eqchar1}. Then
\begin{equation}\label{eqsills4}
\tau=(\tau_1,\tau_2,\cdots, \tau_r,-f_1\cdot 1,-f_2\cdot 3,\cdots , -f_r\cdot (2r-1))\rightarrow (\lambda_1,\lambda_2,\cdots, \lambda_r).
\end{equation}
Thus
$$\lambda_j=\tau_j-f_j-2\sum_{i=j+1}^rf_i.$$
\end{proof}

\begin{example}
Let $\lambda=(20,17,15,12,9,7,4,1)\in GG_1[85]$. Then $\tau_1=20+0+2(5)=30, \tau_2=17+1+2(4)=26$, \ldots, $\tau_8=1+(1+2(0))=2$. So the positive parts of $\tau$ are $(30,26,22,18,14,10,6,2)$ and the negative parts are

$((-1)\wp(20),(-3)\wp(17),\ldots,(-15)\wp(1))=(0,-3,-5,0,-9,-11,0,-15)$. Thus

 $\tau=(30,26,22,18,14,10,6,2,-3,-5,-9,-11,-15).$
\end{example}

\medskip

The second G\"ollnitz--Gordon identity~\cite{G67,G65} states that the number of partitions of $n$
of the type enumerated by $GG_1(n)$ and additionally contain no parts less than $3$, equals the
number of partitions of $n$ into parts congruent to $3$, $4$, or $5$ modulo $8$.

\begin{theorem} \label{SGG}
Let $GG_2(n)$ denote the number of partitions of $n$ with parts at least 3 and and differ by at least $2$ and by at least $4$ if the parts are even.\\
Let $GG_{-2}(n)$ be the number of signed partitions of $n$ in which positive parts are even and at least $2(\ell^+ + 1)$, negative parts are odd, distinct and at most $2\ell^+ $.\\
Let $GG'_{-2}(n)$ denote the number of signed partitions of $n$ in which positive parts are $\geq4$, even and differ by at least $4$, and negative parts are odd, distinct and at most $2\ell^+ - 1$. Then
\begin{equation} \label{GG2}
GG_{2}(n)=GG_{-2}(n)=GG'_{-2}(n).
\end{equation}
\end{theorem}
\begin{proof}

\begin{align}
\sum_{n=0}^\infty GG_2(n)q^n
&= 1+\sum_{n=1}^\infty \frac{q^{n^2+2n} (-q;q^2)_n}{(q^2;q^2)_n}\\
&= 1+\sum_{n=1}^\infty \frac{q^{n^2+2n} (1+q)(1+q^3)\cdots(1+q^{2n-1})}{(1-q^2)(1-q^4)\cdots(1-q^{2n})}\\
&= 1+\sum_{n=1}^\infty \frac{q^{n^2+2n} q^{n^2}(\frac{1}{q}+1)(\frac{1}{q^3}+1)\cdots(\frac{1}{q^{2n-1}}+1)}{(1-q^2)(1-q^4)\cdots(1-q^{2n})}\\
&= 1+\sum_{n=1}^\infty \frac{q^{\overbrace{(2n+2)+\cdots+(2n+2)}^{n\,  \text{times}}}(\frac{1}{q}+1)(\frac{1}{q^3}+1)\cdots(\frac{1}{q^{2n-1}}+1)}{(1-q^2)(1-q^4)\cdots(1-q^{2n})}\\
&= \sum_{n=0}^\infty GG_{-2}(n)q^n\\
&=1+\sum_{n=1}^\infty \frac{q^{4+8+\cdots+(4n)}(\frac{1}{q}+1)(\frac{1}{q^3}+1)\cdots(\frac{1}{q^{2n-1}}+1)}{(1-q^2)(1-q^4)\cdots(1-q^{2n})}\\
&=\sum_{n=0}^\infty GG'_{-2}(n)q^n. 
\end{align}
\emph{Bijective Proofs}. For this 3-way identity, it may be verified that the bijection used in the proof of Theorem \ref{AFGG} is applicable to $GG_2[n]\rightarrow GG_{-2}[n]$ while the bijection used in the proof of Theorem \ref{SFGG} works for $GG_2[n]\rightarrow GG'_{-2}[n]$. 
\end{proof}

\begin{remark}
The difference between the first G\"ollnitz-Gordon partitions (enumerated by $GG_1(n)$) and the 
second  G\"ollnitz-Gordon partitions (enumerated by $GG_2(n)$) 
is the number of partitions of $n$ into parts that mutually differ by $2$ and in which no 
consecutive even numbers appear as parts and exactly one part equals $1$ or $2$.

\begin{align} \sum_{n=0}^\infty \frac{q^{n^2} (-q;q^2)_n}{(q^2;q^2)_n} &- \sum_{n=0}^\infty \frac{q^{n^2+2n} (-q;q^2)_n}{(q^2;q^2)_n} \\ &= \sum_{n=1}^\infty \frac{q^{n^2}(-q;q^2)_n}{(q^2;q^2)_n}( 1 -q^{2n} )\\ &= \sum_{n=1}^\infty \frac{q^{n^2}(-q;q^2)_n}{(q^2;q^2)_{n-1}}\\ 
&= \sum_{n=1}^\infty \frac{q^{2n}\cdot q^{\overbrace{2n+\cdots+2n}^{n-1\,  \text{times}}}(\frac{1}{q}+1)(\frac{1}{q^3}+1)\cdots(\frac{1}{q^{2n-1}}+1)}{(1-q^2)(1-q^4)\cdots(1-q^{2n-2})}\\
&=\sum_{n=1}^\infty \frac{ q^{2+6+\cdots+(4n-1)}(\frac{1}{q}+1)(\frac{1}{q^3}+1)\cdots(\frac{1}{q^{2n-1}}+1)}{(1-q^2)(1-q^4)\cdots(1-q^{2n-2})}
\end{align}

The last two equalities show that the difference is also the number of signed partitions of $N$ in which each positive part is even and $\geq 2\ell^+$ and $2\ell^+$ is a part (or positive parts are even and differ by at least $4$), the negative parts are odd and distinct with each smaller than $2\ell^+$.
\end{remark}

\section{The Little G\"ollnitz identities}\label{gollnitzsmall}
In this section we employ the methods of Section \ref{gollnitz} to give bijective proofs of two theorems involving the two so-called ``little'' G\"ollnitz partition identities (named by K. Alladi).
They appear in G\"ollnitz's paper~\cite{G67}.

\begin{theorem}\label{thm1}
Let $LG_1(n)$ be the number of partitions of $n$ into parts that differ by at least 2 with odd parts differing by at least 4 (i.e., set of first little G\"ollnitz partitions).\\
Let $E(n)$ be the number of signed partitions of $n$ with $k$ positive parts which are even and distinct and $t$ negative parts which are odd, distinct and less than $2k$ such that the smallest positive part is greater than $2t-\delta_{1u}$, where $u$ is the smallest negative part. 
Then $LG_1(n) = E(n)$.
\end{theorem}

\begin{proof} The proof is analogous to that of Theorem \ref{SFGG}.

Every $\lambda\in LG_1[n]$ has the form $\lambda = (\lambda_1,\lambda_2,\ldots,\lambda_k)$, where $\lambda_j>\lambda_{j+1}+1$ such that when both parts are odd, then $\lambda_j>\lambda_{j+1}+2$.
But each $\pi\in E[n]$ may be expressed as 

$\pi = (\pi_1,\ldots,\pi_k,-u_1,-u_2,\ldots,-u_t)$, where $\pi_1 > \ldots > \pi_k > -u_1 > \cdots > -u_t,\, t\geq k\geq 1\ $ such that $\ \pi_i \equiv 0,\, u_j\equiv 1$ (mod 2) $\forall\, i, j$ with $\pi_k>2t-\delta_{1u_1}$.

Define a map $\phi: LG_1[n]\rightarrow E[n]$. If $\phi(\lambda)=\pi\in E[n]$, then we have
\begin{equation}\label{eqn3}
\pi = (\pi_1,\pi_2,\cdots,\pi_k,-1\wp(\lambda_k),-3\wp(\lambda_{k-1}), \cdots,-(2k-1)\wp(\lambda_1)),
\end{equation}
where 
\begin{equation}\label{eqn4}
\pi_j=\lambda_j+\wp(\lambda_j)+2\sum_{i=1}^{j-1}\wp(\lambda_i),\ 1\leq j\leq k.
\end{equation}
We see that $\pi_j\equiv 0$ (mod 2) for all $i$, and if $\lambda$ has no negative parts, then $\pi = \lambda$. 

It is clear that the number of odd parts of $\lambda$ is equal to the number of negative parts of $\pi$. Moreover, for any index $j$, we have $\pi_j - \pi_{j+1} = \lambda_j - \lambda_{j+1} - \wp(\lambda_j) - \wp(\lambda_{j+1})\geq 2$ since $\lambda_j - \lambda_{j+1}\geq 2,3,4$ and $\wp(\lambda_j) + \wp(\lambda_{j+1}) = 0,1,2$ when $\lambda_j$ and $\lambda_{j+1}$ are both even, have opposite parities or both odd, respectively.

The inverse map $\phi^{-1}$ is obtained as follows. 
Let $\pi=(\pi_1,\pi_2,\cdots, \pi_k,-f_1\cdot 1,-f_2\cdot 3,\cdots , -f_k\cdot (2k-1))\in E[n]$, where the $f_j$ are defined for $\pi$ as in \eqref{eqchar1}. Then
\begin{equation}\label{eqsills4a}
\pi \longmapsto (\lambda_1,\lambda_2,\cdots, \lambda_k),
\end{equation}
where 
\begin{equation}\label{eqsills4b}
\lambda_j=\pi_j-f_{k-j+1}-2\sum_{i=1}^{j-1}f_{k-i+1}, \ 1\leq j\leq k.
\end{equation}

Finally, note that 
$$\max(\pi_j-\lambda_j\mid 1\leq j\leq k)= f_{1} + 2\sum_{i=1}^{k-1}f_{k-i+1} = f_{1} + 2\sum_{i=2}^{k}f_{i} = 2t-\delta_{1u},$$
where $t$ is the number of odd parts and $u$ is the smallest odd part size of $\pi$.
But the smallest part of $\lambda$ which is $\lambda_k$ satisfies
$\lambda_k=\pi_k-(2t-\delta_{1u})>0.$
\vskip  4pt

For example consider $\lambda = (31,26,24,21,17,14,11,7,4)\in LG_1[155]$.  
Then using \eqref{eqn3} and \eqref{eqn4} we obtain
$\phi(\lambda) = \pi =(32, 28, 26, 24, 22, 20, 18, 16, 14,0, -3, -5, 0, -9, -11, 0, 0, -17)\in E[155].$

Conversely, the pre-image of $\pi$ may be similarly recovered using \eqref{eqsills4a} and \eqref{eqsills4b}.
\end{proof}

The following is a dual identity to Theorem \ref{thm1} that corresponds to the second little 
G\"ollnitz idenity. The assertion is analogous to Theorem \ref{thm1} and may be proved in the same manner.

\begin{theorem}\label{thm2z}
Let $LG_2(n)$ be the number of partitions of $n$ without 1's in which parts differ by at least 2 with odd parts differing by at least 4 (i.e., the second little G\"ollnitz partitions).\\
Let $T(n)$ be the number of signed partitions of $n$ with $k$ positive parts which are even and distinct and $t$ negative parts which are odd, distinct and less than $2k$ such that the smallest positive part is greater than $2t+\delta_{1u}$, where $u$ is the smallest odd part.
Then $LG_2(n) = T(n)$.
\end{theorem}

Observe that Theorem \ref{thm2z} satisfies $\pi_k-2\geq 2t-\delta_{1u}$ or $\pi_k > 2t + 1 - \delta_{1u}$.  

\bigskip

\begin{theorem}
Let $LG_1(n)$ denote the number of partitions of $n$ in which the difference between parts are at least $2$ and at least $4$ if the parts are odd. Let $LG_{-1}(n)$ denote the number of signed partitions of $n$ in which positive parts are at least $5$, odd and differ by at least $4$ and negative parts are distinct, $\equiv 1 \pmod 2$, and at most $2\ell^++1$. Then
\begin{equation} \label{An1}
LG_1(n) = LG_{-1}(n-1).
\end{equation} 
\end{theorem}
\begin{proof}
    The present authors~\cite{AMS25} showed that 
\begin{equation}\label{AMS1}
    \sum_{n=0}^\infty LG_1(n)q^n=\sum_{n=0}^\infty \frac{q^{n(n+1)}(-q^{-1};q^2)_n}{(q^2;q^2)_{n}}=\sum_{n=0}^\infty \frac{q^{n(n+1)}(-q;q^2)_{n+1}}{(q^2;q^2)_{n}}.
\end{equation}
Therefore, from~\ref{AMS1} we get
\begin{align}
    \sum_{n=0}^\infty \frac{q^{n(n+1)}(-q;q^2)_{n+1}}{(q^2;q^2)_{n}}&=\sum_{n=0}^\infty \frac{q^{2+4+\cdots+2n}(-q;q^2)_{n+1}}{(q^2;q^2)_{n}}\\
    &=\sum_{n=0}^\infty \frac{q^{(n+1)^2}q^{2+4+\cdots+2n}(-q^{-1};q^{-2})_{n+1}}{(q^2;q^2)_{n}}\\
    &=\sum_{n=0}^\infty\frac{q\cdot q^{3+5+\cdots+(2n+1)}q^{2+4+\cdots+2n}(-q^{-1};q^{-2})_{n+1}}{(q^2;q^2)_{n}}\\
    &=\sum_{n=0}^\infty\frac{q\cdot q^{5+9+\cdots+(4n+1)}(-q^{-1};q^{-2})_{n+1}}{(q^2;q^2)_{n}}\\
    &=q\sum_{n=0}^\infty LG_{-1}(n)q^n
\end{align}

\end{proof}

\begin{theorem}\label{thmflg}
 Let $LG_1(n)$ denote the number of partitions of $n$ in which the difference between parts are at least $2$ and at least $4$ if the parts are odd. Let $LG'_{-1}(n)$ denote the number of signed partitions of $n$ in which positive parts are distinct and even and negative parts are distinct, $\equiv 1 \pmod 4$, and at most $2\ell^+$. Then
\begin{equation} \label{An1}
LG_1(n) = LG'_{-1}(n).
\end{equation}   
\end{theorem}
\begin{proof}
 In~\cite{SS11} Savage and the third author derived new version of the first
 little G\"ollnitz identity 
 and they showed that $$\sum_{n=0}^\infty LG_1(n)q^n=\sum_{n=0}^\infty \frac{q^{2n^2-n}(-
 q;q^4)_n}{(q^2;q^2)_{2n}}.$$ 
Therefore, 
 
 \begin{align}
\sum_{n=0}^\infty \frac{q^{2n^2-n} (-q;q^4)_n}{(q^2;q^2)_{2n}}
&= 1+\sum_{n=1}^\infty \frac{q^{2n^2-n} (-q;q^4)_n}{(q^2;q^2)_{2n}}\\
&= 1+\sum_{n=1}^\infty \frac{q^{2n^2-n} (1+q)(1+q^5)\cdots(1+q^{4n-3})}{(1-q^2)(1-q^4)\cdots(1-q^{4n})}\\
&= 1+\sum_{n=1}^\infty \frac{q^{2n^2-n} q^{2n^2-n}(\frac{1}{q}+1)(\frac{1}{q^5}+1)\cdots(\frac{1}{q^{4n-3}}+1)}{(1-q^2)(1-q^4)\cdots(1-q^{4n})}\\
&= 1+\sum_{n=1}^\infty \frac{q^{0+2+\cdots+2(2n-1)}(\frac{1}{q}+1)(\frac{1}{q^5}+1)\cdots(\frac{1}{q^{4n-3}}+1)}{(1-q^2)(1-q^4)\cdots(1-q^{4n})}\\
&=\sum_{n=0}^\infty LG'_{-1}(n)q^n .
\end{align}

\end{proof}

\begin{theorem} \label{SLG}
 Let $LG_2(n)$ denote the number of partitions of $n$ into parts $\geq 2$ and the difference between parts are at least $2$ and at least $4$ if the parts are odd. Let $LG_{-2}(n)$ denote the number of signed partitions of $n$ where the positive parts are odd and at least $2\ell^+$, the negative parts are odd, distinct and at most $2\ell^+$. Let $LG'_{-2}(n)$ denote the number of signed partitions of $n$ in which positive parts are $\geq3$, odd and differ by $4$, and negative parts are odd, distinct and at most $2\ell^+-1$. Then
\begin{equation} \label{LG2x}
LG_2(n) = LG_{-2}(n)=LG'_{-2}(n).
\end{equation}
\end{theorem}   

\begin{proof} We have the following.    
\begin{align} \sum_{n=0}^\infty LG_2(n)q^n &= 1+\sum_{n=1}^\infty \frac{q^{n^2+n} (-q;q^2)_n}{(q^2;q^2)_n}\\ &= 1+\sum_{n=1}^\infty \frac{q^{n^2+n} (1+q)(1+q^3)\cdots(1+q^{2n-1})}{(1-q^2)(1-q^4)\cdots(1-q^{2n})}\\ &= 1+\sum_{n=1}^\infty \frac{q^{n^2+n} q^{n^2}(\frac{1}{q}+1)(\frac{1}{q^3}+1)\cdots(\frac{1}{q^{2n-1}}+1)}{(1-q^2)(1-q^4)\cdots(1-q^{2n})}\\ &= 1+\sum_{n=1}^\infty \frac{q^{\overbrace{(2n+1)+(2n+1)+\cdots+(2n+1)}^{n \text{times}}}}{(1-q^2)(1-q^4)\cdots(1-q^{2n})}\\ &\times (\frac{1}{q}+1)(\frac{1}{q^3}+1)\cdots(\frac{1}{q^{2n-1}}+1)\\
&= \sum_{n=0}^\infty LG_{-2}(n)q^n\\
&= 1+\sum_{n=1}^\infty \frac{q^{3+7=\cdots+(4n-1)}}{(1-q^2)(1-q^4)\cdots(1-q^{2n})}\\ &\times (\frac{1}{q}+1)(\frac{1}{q^3}+1)\cdots(\frac{1}{q^{2n-1}}+1)\\
&= \sum_{n=0}^\infty LG'_{-2}(n)q^n.
\end{align}

Bijective Proofs of the 3-way identity \eqref{LG2x}. A bijection $LG_2[n] \rightarrow LG_{-2}[n]$ may be establlished using the proof of Theorem \ref{AFGG} while the bijection in Theorem \ref{SFGG} also works for $LG_2[n]\rightarrow LG'_{-2}[n]$, provided that the parity function $\wp(k)$ in both bijections is  modified to
  \[ \wp(k)  = \begin{cases}
  0 & \mbox{ if $ \text{ $k$ is odd} $}\\
  1 & \mbox{ if \text{ $k$ is even}}.
  \end{cases}
  \]
  
\end{proof}

\begin{theorem}\label{thm3m4}
Let $H(n)$ be the set of signed partitions of $n$ in which positive parts are even and distinct, negative parts are $\equiv 3$ (mod 4) and less than twice the number of positive parts.
Then $LG_2[n] = H[n]$.
\end{theorem}

\begin{proof}
  In~\cite{SS11} Savage and Sills derived a second new little G\"ollnitz identity and they showed that 
 $$\sum_{n=0}^\infty LG_2(n)q^n=\sum_{n=0}^\infty \frac{q^{2n^2+n}(-q^{-1};q^4)_n}
 {(q^2;q^2)_{2n}}.$$ 
  By using the same manipulation that we used in Theorem~\ref{thmflg}, we will get the required result.
\end{proof}

\section{Conclusion}
This work demonstrates a natural method for interpreting $q$-series in terms of signed partitions;
several examples of classical $q$-series are exploited for this purpose.   The interested reader
could employ the techniques here to provide analogous interpretations for other $q$-series
such as those in the paper of Slater~\cite{S52}.

\end{document}